\documentclass[leqno, english]{amsart}
\usepackage{etex}
\usepackage{amsmath,amssymb,amsthm,bbm,mathtools,comment}
\usepackage[shortlabels]{enumitem}
\usepackage[pdftex,colorlinks,backref=page,citecolor=blue]{hyperref}
\hypersetup{pdfpagemode=UseNone,pdfstartview={XYZ null null 1.00}}
\usepackage[mathscr]{euscript}
\usepackage[usenames,dvipsnames]{xcolor}
\usepackage{adjustbox,tikz,calc,graphics,babel,standalone}
\usetikzlibrary{shapes.misc,calc,intersections,patterns,decorations.pathreplacing}
\usetikzlibrary{arrows,shapes,positioning}
\usetikzlibrary{decorations.markings}
\usepackage[final]{microtype}
\usepackage[numbers]{natbib}
\usepackage{amsfonts}
\usepackage{graphicx}
\usepackage{caption}
\usepackage{subcaption}
\usepackage{verbatim}
\usepackage{array}
\usepackage[frame,cmtip,arrow,matrix,line,graph,curve]{xy}
\usepackage{graphpap, color, pstricks}
\usepackage{pifont}
\usepackage[final]{microtype}
\usepackage{cmtiup}

\allowdisplaybreaks

\theoremstyle{plain}
\newtheorem{theorem}{Theorem}[section] 
\newtheorem{lemma}[theorem]{Lemma}
\newtheorem{claim}[theorem]{Claim}
\newtheorem{proposition}[theorem]{Proposition}
\newtheorem{corollary}[theorem]{Corollary}
\newtheorem{conjecture}[theorem]{Conjecture}
\newtheorem{problem}[theorem]{Problem}
\newtheorem{question}[theorem]{Question}
\theoremstyle{remark}

\theoremstyle{definition}

\def\N{\mathbb{N}}
\newcommand{\gd}{\delta}
\newcommand{\cC}{\mathcal{C} }
\newcommand{\cF}{\mathcal{F} }

\newcommand{\R}[2]{\text{r}_{#2}}
\newcommand{\K}[2]{K_{#1}^{(#2)}}
\newcommand{\He}[3]{\text{H}_{#1}^{(#2)}(#3)}
\newcommand{\lef}{\text{left}}
\newcommand{\ri}{\text{right}}

\newcommand{\eps}{\ensuremath{\varepsilon}}
\let\emptyset\varnothing

\let\originalleft\left
\let\originalright\right
\renewcommand{\left}{\mathopen{}\mathclose\bgroup\originalleft}
\renewcommand{\right}{\aftergroup\egroup\originalright}

\makeatletter
\def\imod#1{\allowbreak\mkern10mu({\operator@font mod}\,\,#1)}
\makeatother

\title{Tower Gaps in Multicolour Ramsey Numbers}
\author[Q. Dubroff]{Quentin Dubroff}
\address{Department of Mathematics, Rutgers University, Piscataway, NJ 08854, USA} \email{qcd2@math.rutgers.edu}

\author[A. Gir\~ao]{Ant\'onio Gir\~ao}
\address{Mathematical Institute, University of Oxford, Oxford OX2 6GG, UK}
\email{antonio.girao@maths.ox.ac.uk} 

\author[E. Hurley]{Eoin Hurley}
\address{Korteweg-de Vries Institute for Mathematics, Universiteit van Amsterdam, Amsterdam, Netherlands}
\email{e.p.hurley@uva.nl}

\author[C. Yap]{Corrine Yap}
\address{School of Mathematics, Georgia Institute of Technology, Atlanta, GA 30308, USA} \email{math@corrineyap.com}

\begin{document}

\maketitle
\begin{abstract}

Resolving a problem of Conlon, Fox, and R\"{o}dl, we construct a family of hypergraphs with arbitrarily large tower height separation between their $2$-colour and $q$-colour Ramsey numbers. The main lemma underlying this construction is a new variant of the Erd\H{o}s--Hajnal stepping-up lemma for a generalized Ramsey number $r_k(t;q,p)$, which we define as the smallest integer $n$ such that every $q$-colouring of the $k$-sets on $n$ vertices contains a set of $t$ vertices spanning fewer than $p$ colours. Our results provide the first tower-type lower bounds on these numbers.
\end{abstract}

\section{Introduction}

Let $K_n^{(k)}$ denote the complete $k$-uniform hypergraph on $n$ vertices. 
We define $r_k(G; q)$ for $k, q \in \mathbb N$ as the smallest integer $n$ such that in every $q$-colouring of $K_n^{(k)}$, there is a monochromatic copy of the hypergraph $G$. For simplicity when $G$ is $\K{t}{k}$, we write $r_k(G;q) = r_k(t;q)$. 
Observe that when $q = 2$, $r_k(G;2)$ and $r_k(t;2)$ coincide with the classical Ramsey numbers $r_k(G)$ and $r_k(t)$, and we will denote them as such.
One of the most central open problems in Ramsey theory is determining the growth rate of  the 3-uniform Ramsey number $r_3(t)$. 
A famous result of Erd\H{o}s, Hajnal, and Rado~\citep{EHR} from the 60's shows that there exist constants $c$ and $c'$ such that
\[
2^{ct^2} \leq r_3(t) \leq 2^{2^{c't}}.
\]
Note that the upper bound is essentially exponential in the lower bound. Despite much attention, this remains the state of the art. 
Perhaps surprisingly, if we allow four colours instead of two, Erd\H{o}s and Hajnal (see e.g.\ \citep{grs}) showed that the double-exponential upper bound is essentially correct, i.e.\ there is a $c>0$ such that $r_3(t;4)\geq 2^{2^{ct}}$.
More recently Conlon, Fox, and Sudakov~\cite{CFS-JAMS} proved a super-exponential bound with three colours, that is, that there exists $c>0$ such that $r_3(t;3) \geq 2^{t^{c\log t}}$.
 Erd\H{o}s conjectured that the double-exponential bound should hold without using extra colours, offering \$500 dollars for a proof that $r_3(t) \geq 2^{2^{ct}}$ for some constant $c>0$. 
Raising the stakes for this conjecture is the ingenious \textit{stepping-up} construction of Erd\H{o}s and Hajnal (see e.g.\ \citep{grs}), which shows that for all $q$ and $k\geq 3$,
\begin{equation}\label{eq:classic_stepping_up}
    r_{k+1}(2t+k-4;q) > 2^{r_{k}(t;q) - 1}.
\end{equation}
For the past 60 years, we have used \eqref{eq:classic_stepping_up} to stack our lower bounds for $r_k(t;q)$ upon that of  $r_3(t;q)$,  yielding that $r_k(t) \geq T_{k-1}(ct^2)$, where \emph{$T_k(x)$, the tower of height $k$ in $x$,} is defined by $T_1(x) = x,\, T_{i+1}(x) = 2^{T_i(x)}$.
The corresponding upper bounds of $r_k(t) \leq T_k(O(t))$ (see~\citep{EH,ER,EHR}) are once again exponential in the lower bounds,
and thus a positive resolution of Erd\H{o}s's conjecture would be the decisive step in showing that $r_k(t) = T_k(\Theta(t))$ for all $k\geq 3$.

Due to the lack of progress on this central conjecture, it is natural to try to understand just how significant a role the number of colours can play in hypergraph Ramsey numbers and whether or not there could really be such a large difference between $r_3(t)$ and $r_3(t;4)$.  
One argument in favour of the conjecture is that the reliance on extra colours to prove a double-exponential lower bound may be a technical limitation of the stepping-up construction.
This is challenged by a stunning discovery of Conlon, Fox, and R\"odl~\citep{CFR} who exhibited an infinite family of $3$-uniform hypergraphs called \textit{hedgehogs}, whose Ramsey numbers display strong dependence on the number of colours.
Namely, they showed that the $2$-colour Ramsey number of hedgehogs is polynomial in their order, while the $4$-colour Ramsey number is at least exponential. 
To understand just how significant a role the number of colours could play they asked the following:
\begin{question}\label{q:large_gaps}
For any integer $h \geq 3$, do there exist integers $k$ and $q$ and a family of $k$-uniform hypergraphs for which the $2$-colour Ramsey number grows as a polynomial in the number of vertices, while
the $q$-colour Ramsey number grows as a tower of height $h$?
\end{question}

Our main contribution is to answer this in the affirmative. 
Define the \emph{$k$-uniform balanced hedgehog} $\hat H^{(k)}_{t}$ with body of order $t$ to be the graph constructed as follows: take a set $S$ of $t$ vertices, called the {\em body}, and for each subset $X\subset S$ of order $\lceil \frac k2 \rceil$ add a $k$-edge $e$ with $e\cap S = X$ such that for all $e,f \in E(\hat H^{(k)}_{t})$ we have $e\cap f\subset S$. 
The hedgehog $H_t^{(k)}$ as defined by Conlon, Fox, and R\"odl differs only in that they consider every $X\subset S$ of order $k-1$ rather than $\lceil \frac k2 \rceil$.  
We observe that for $k=3$ the two definitions coincide.
When the uniformity is clear from the context we shall drop the superscript.
\begin{theorem}\label{thm:hedgehogs_upper}
There exist $c>0$ and $q:\N \rightarrow \N$ such that for all $k\in \N$ and sufficiently large $t$, we have
\begin{enumerate}[(a)]
    \item \label{itm:hedgehog_upper} $ r_{2k+1}(\hat H_{t}) \leq t^{k+3}$,  and
    \item \label{itm:hedgehog_lower}  $r_{2k+1}(\hat H_{t};q(k)) \geq T_{\lfloor c\log_2\log_2 k\rfloor}(t)$ .
\end{enumerate}
\end{theorem}

To prove this, we provide new stepping-up lemmas for a more general type of hypergraph Ramsey numbers. Let $r_k(G;q,p)$ for $q \geq p$ be the smallest integer $n$ such that in every $q$-colouring of $K_n^{(k)}$, there is a copy of the hypergraph $G$ whose edges span fewer than $p$ colours. 
As before, we use $r_k(t; q, p)$ when $G = \K{t}{k}$ and suppress $p$ when $p = 2$. 

A standard application of the first moment method (see e.g.\ \cite{alonspencer}) shows that for any $k,q\in \N$ there exists $c>0$ such that  $r_k(t;q,q) \geq 2^{ct^{k-1}}$ for all $t\in \N$. 
We note that in the graph case ($k=2$) the special case of $q = p$ was already investigated by Erd\H{o}s and Szemer\'{e}di~\cite{ESz} in the 70's; in fact, the more general case when $p<q$ is also indirectly discussed. They showed the following rather precise bounds: for all $q\ll t$, $2^{\Omega(t/q)} \leq r_2(t; q, q) \leq q^{O(t/q)}$ . 

These generalized hypergraph Ramsey numbers were also considered in a special case by Conlon, Fox, and R\"odl \citep{CFR} who asked if there exist an integer $q$ and number $c>0$ such that $r_3(t;q,3)\geq 2^{2^{ct}}$.
To date, the only nontrivial improvement on the first moment bound has been made by
Mubayi and Suk~\citep{MS} who proved there exists $c>0$ such that for $q \geq 9$, we have $r_3(t;q,3)\geq 2^{t^{2+cq}}$ for $t\in \N$ sufficiently large;
for all other values of $k,q,p \geq 3$, the random construction is essentially the state of the art. 
Our knowledge (or lack thereof) is thus summarised by the following bounds for $k,q,p \geq 3$ and sufficiently large $t\in \N$,
\[
2^{t^c} \leq r_k(t;q,p) \leq T_k(O(t)),
\]
where $c\geq 1$ is allowed to depend on $k,q$ and $p$. 
Note that in this case our upper bounds are a staggering tower of height $k-2$ in the lower bounds.

A related notion called the set-colouring Ramsey number was introduced by Erd\H{o}s, Hajnal, and Rado in~\cite{EHR} and subsequently studied in~\cite{setramsey2009} and much more recently in~\cite{setramsey} and~\cite{setramsey2}.
Borrowing notation from \cite{setramsey}, let $R_k(t;q,s)$ denote the minimum number of vertices such that every {\em $(q,s)$-set colouring} of $K_n^{(k)}$, that is, a colouring in which each $k$-set is assigned an element of ${[q] \choose s}$, contains a monochromatic $\K{t}{k}$. 
Here, monochromatic means the intersection of the colour sets assigned to the edges is nonempty. Observe that certain cases of $R_k$ and $r_k$ coincide.
For example, $R_k(t;q,q-1) = r_k(t;q,q)$ and in general, we have the bound
\[r_{k}(t ; q, p) \leq R_k\left(t ;{q \choose p-1}, {q-1 \choose p-2}\right).\]
We prove lower bounds on $r_k(t; q, p)$, thus giving lower bounds on certain set-colouring Ramsey numbers. However, we are not able to definitively resolve any questions from \cite{setramsey}, due to central gaps in our understanding of hypergraph Ramsey numbers. See Section~\ref{sec:conc} for more on this.

Our main tool in the proof of Theorem~\ref{thm:hedgehogs_upper} is the development of two new stepping-up constructions 
which yield the first tower-type results of their kind. 
We show the following three stepping-up statements, listed in order of decreasing strength, with ${C_k = \frac{
1}{k+1}{2k \choose k}}$ denoting the $k$-th Catalan number.
\begin{theorem}\label{thm:step-up}
Let $k, q, p \geq 3$. There exist $c \geq 1$ and $t_0$ such that for all $t > t_0$,
\begin{enumerate}[(a)]
    \item if $p \leq C_k-2$, then $r_{k+1}(t^c; q,p) > 2^{r_k(t;q,p)-1}$, \label{thm:step-up1}
    \item if $p \leq C_k$, then $r_{k+1}(t^c; 2q+p,p) > 2^{r_k(t;q,p)-1}$,\label{thm:step-up2} and
    \item if $p \leq k!$, then $r_{2k}(t^c;qp,p) > 2^{r_k(t;q,p)-1}$. \label{thm:step-up3}
\end{enumerate}
\end{theorem}

Note that the growth rate in $k$ which is implied by Part~\ref{thm:step-up3} (approximately a tower of height $\log_2 k$) of Theorem~\ref{thm:step-up} is much smaller than that of Parts~\ref{thm:step-up1} and~\ref{thm:step-up2} because we can only step up at the cost of doubling the uniformity size.  
Unfortunately, this does not allow us to answer the question of Conlon, Fox, and R\"odl on $r_3(t;q,3)$, since $C_2 = 2$, but already for $k\geq 4$ we have the following two corollaries:
\begin{corollary}\label{cor:cliquescolours}
For all $k\geq 4$, there is $q\in \N$ and $c>0$ such that $
r_k(t;q,5) \geq T_{k-1}(t^{c})$.
\end{corollary}
\begin{corollary}\label{cor:(3,3)}
For all $k\geq 4$, there is $c>0$ such that $
r_k(t;3,3) \geq T_{k-1}(t^{c})$.
\end{corollary}

Observe that by the second corollary the growth rate of $r_k(t;3,3)$ matches the current best lower bounds for $r_k(t)$ up to a polynomial in $t$. 
The reason we have an absolute constant $c$ in the exponent is due to the use of an Erd\H{o}s-Hajnal type result on sequences (see Section~\ref{sec:indEHfam}).

The second main element of our proof connects the problem of avoiding monochromatic balanced hedgehogs to that of avoiding cliques that span few colours.  
It is a straightforward adaptation of ideas from Conlon, Fox, and R\"odl~\cite{CFR}. 

\begin{lemma}\label{lem:hedgehogs_construction}
 Given $k,q,t \in \N$, let $p={2k+1 \choose k+1}$ and $q' = {q \choose p}$. Then
 \[
 r_{2k+1}(\hat H_{t};q',2) > r_{k+1}\left(t; q, p+1\right) - 1.
 \]
\end{lemma}

Using this result along with Part~\ref{thm:step-up3} of Theorem~\ref{thm:step-up} yields the lower bound in Theorem \ref{thm:hedgehogs_upper}\ref{itm:hedgehog_lower}.
It is natural to ask whether one can combine the growth rate in $k$ given by Part~\ref{thm:step-up1} of Theorem~\ref{thm:step-up} with the ability to impose as many colours as in Part~\ref{thm:step-up3}. Unfortunately, the condition $p \leq C_k$ prevents us from using Part~\ref{thm:step-up1} as the right-hand side because $C_k = \frac{1}{k+1}{2k\choose k} < {2k+1 \choose k+1}$. 
This is tantalisingly close, if not a little curious, as the dependence on $C_k$ comes from our exact solution to a subsequence avoidance problem (Corollary~\ref{cat}). 
We show that $C_k$ presents a natural barrier in this endeavour. 
This barrier is made concrete by some new and tight results on the Ramsey theory of sequences, including an Erd\H{o}s-Hajnal-type result in Section~\ref{sec:indEHfam}.

\subsection{Outline}
The outline of the paper is as follows: in Section~\ref{sec:indEHfam} we introduce and prove results on the Ramsey theory of sequences; in Section~\ref{sec:app} we use these to prove our stepping up constructions, namely, Theorem~\ref{thm:step-up}; and in Section~\ref{sec:hedgehogs} we prove our hedgehog-related results, including Theorem~\ref{thm:hedgehogs_upper} and  Lemma~\ref{lem:hedgehogs_construction}, and provide a construction of a degenerate hypergraph that relates to the Burr-Erd\H{o}s Conjecture.
To finish, we pose some questions and problems highlighted by our results.

%%%%%%%%%%%%%%%%%%%%%%%%%%%%%%%%%%%%%%%%

%%%%%%%%%%%%%%%%%%%%%%%%%%%%%%%%%%%%%%%

\section{Ramsey Theory of Sequences}\label{sec:indEHfam}

Let $V=\{0,1\}^m$ and  given vectors $v,w \in V$, define
\[\gd(v,w) = \max \{i: v_i \neq w_i\}.\]
We say $v<w$ if $v_\gd < w_\gd$. For every set of vertices $v_1< v_2 < \dots < v_{k+1}$ in $V$ there is a corresponding sequence $(\gd_1,\gd_2,\ldots,\gd_{k})$ given by $\gd_i = \gd(v_i,v_{i+1})$. 
In this section, we introduce some definitions that will be useful in order to analyse the structure of these $\gd$-sequences.

We say a sequence $S = (a_1, a_2, \dots, a_m)$ is \textit{monotonic} if $a_1 \leq a_2\leq \dots \leq a_m$ or $a_1\geq a_2\geq \dots \geq a_m$. Two sequences $(a_1, a_2, \ldots,a_t)$ and $(b_1, b_2, \ldots, b_t)$ have the same pattern if the relative ordering of every pair of elements is the same, i.e.  $a_i>a_j$, $a_i = a_j$, or $a_i < a_j$ if and only if $b_i>b_j$, $b_i = b_j$, or $b_i < b_j$ respectively for all $1 \le i < j \le t$. A \emph{pattern} is then the equivalence class of sequences with respect to this relative ordering.
Given sequences $S$ and $P$, we say $S$ \textit{contains the pattern $P$} if we can find a subsequence of $S$ which has the same pattern as $P$; otherwise, we say $S$ \textit{avoids} $P$. 
A pattern is a {\em permutation pattern} if there is a permutation in the equivalence class, and we often use this permutation as a representative for the pattern. In what follows, we will mainly be concerned with permutations that avoid the patterns $132$ or $231$.

%%%%%%%%%%
\subsection{Max-induced Subsequences}
With the aforementioned applications to hypergraph Ramsey theory in mind, we introduce ``max-induced" subsequences. 
We say $(a_{i_1},a_{i_2},\ldots,a_{i_t})$ with $i_1<i_2<\dots <i_t$ is a \textit{max-induced subsequence} of $(a_1,a_2,\ldots, a_m)$ if the maximum of $(a_{i_j},a_{i_j+1},\ldots, a_{i_{j+1}})$ is attained at $a_{i_j}$ or $a_{i_{j+1}}$, i.e. at the left or right extreme, for all $1 \leq j < t$. We say a sequence $S$ contains a \textit{max-induced pattern} $P$ if there is a max-induced subsequence of $S$ which has pattern $P$, and a family of sequences $\cF$ has the \textit{max-induced Erd\H{o}s-Hajnal} property with exponent $c(\cF)$ if any sequence $S$ that avoids every member of $\cF$ as a max-induced subsequence has a monotonic max-induced subsequence of order $|S|^{c(\cF)}$, where $|S|$ denotes the length of the sequence. We are able to characterize the families with this property.

The following fact, whose short proof can be found in Section~\ref{sec:app}, relates the max-induced property to our original motivation:
\begin{claim}\label{CorrespEdge}
Suppose $v_1 < v_2 < \dots < v_{\ell+1}$ are vectors in $\{0,1\}^m$ and let $\delta_i = \delta(v_i, v_{i+1})$. If $(\gd_{i_1},\ldots, \gd_{i_k})$ is a max-induced subsequence of $(\gd_1,\ldots,\gd_\ell)$, then there are $v_{j_1},\ldots, v_{j_{k+1}}$ such that $\gd(v_{j_s}, v_{j_{s+1}}) =  \gd_{i_s}$ for each $s\in[k]$.
\end{claim}

Define an {\em interval} in a sequence to be a subsequence of consecutive elements. 
The following theorem is our core result on max-induced patterns:  
\begin{theorem}\label{thm:indEHfam}
A finite family of patterns $\mathcal{F}$ has the max-induced Erd\H{o}s-Hajnal property if and only if $\mathcal F$ contains a $132$-avoiding permutation pattern and a $231$-avoiding permutation pattern.
\end{theorem}

Observe that if a single  permutation is $\{132, 231\}$-avoiding, it must be decreasing up to some element and increasing for the rest of the permutation. We say such a permutation has a {\em unique local minimum}, and we immediately get the following corollary of Theorem \ref{thm:indEHfam}:
\begin{corollary}\label{cor: maxEH}
A permutation $P$ has the max-induced Erd\H{o}s-Hajnal property (with exponent $c(P) = 16^{-|P|}$) if and only if it has a unique local minimum (in other words it is $\{132, 231\}$-avoiding).
\end{corollary}
An old result of Shelah~\citep{shelah} states that any graph which does not have large cliques or large independent sets must contain exponentially many non-isomorphic induced subgraphs. 
Taking inspiration from this, we may ask: given $k \in \mathbb N$, what is the largest integer $f(k)$ for which there is $\varepsilon>0$ such that every sequence of length $n$ either contains at least $f(k)$ distinct max-induced patterns on $k$ elements or a max-induced monotonic subsequence of length $n^{\varepsilon}$? 
It is well-known that the number of 132-avoiding permutations on $[k]$ is exactly the Catalan number, $C_k$ (see e.g.~\cite[Section 1.5]{stanley}). Thus, Theorem \ref{thm:indEHfam} allows us to answer the above question exactly and is the origin of the Catalan number that appears in Theorem~\ref{thm:step-up}. 

\begin{corollary}\label{cat}
The value of $f(k)$ is the number of $132$-avoiding permutations (equivalently $231$-avoiding), which is equal to the Catalan number $C_k$.
\end{corollary}
To prove Theorem~\ref{thm:indEHfam}, we require the following lemma:

\begin{lemma}\label{231_avoiding}
For every $k \geq 1$, there is a $231$-avoiding permutation $S_k$ of length $2^{k+1}-1$ that 
does not contain a max-induced monotonic subsequence of length greater than $k+1$.
\end{lemma}

\begin{proof}
We proceed by induction on $k$. Note that for $k=1$, the permutation $(1,3,2)$ works. Let $S_{k-1} = (a_1, a_2, \dots, a_{m_{k-1}})$ be a $231$-avoiding permutation of length $2^{k}-1$ with
no max-induced monotonic subsequence of length greater than $k$. Define the sequence $S_k = (b_1, \dots, b_{2^{k+1}-1})$ by 
$$b_i =
\begin{cases}
a_i &\text{if } i < 2^k; \\
2^{k+1}-1 &\text{if } i = 2^k;\\
a_{i-2^k} + 2^{k}-1 &\text{if } i > 2^k.\\
\end{cases}
$$

\noindent
If $(b_{i_1},\ldots , b_{i_\ell})$ is a max-induced monotonic subsequence of $S_k$, then either $i_\ell \leq 2^k$ or $i_1 \geq 2^k$. Thus, by our choice of $S_{k-1}$, there is no max-induced monotonic component of length greater than $k+1$ in $S_k$. Similarly, any possible $231$ pattern $(b_{i_1}, b_{i_2}, b_{i_3})$ has either $i_3 < 2^k$ or $i_1 > 2^k$ by construction, so the inductive hypothesis shows $S_k$ is $231$-avoiding.
\end{proof}

\begin{proof}[Proof of Theorem~\ref{thm:indEHfam}]
For the forward direction, suppose (without loss of generality) that $\cF$ contains no $231$-avoiding permutation pattern. 
Then each sequence in $\cF$ contains either the pattern $231$ or two elements which are equal. The sequence $S_k$ defined in Lemma~\ref{231_avoiding} is an $\mathcal F$-avoiding sequence of length at least $2^k$ containing no max-induced monotonic subsequence of length greater than $k+1$.

For the backwards direction, let $L$ and $R$ be two permutations such that $L$ is $132$-avoiding and $R$ is $231$-avoiding, and suppose $S$ is a sequence of length $n$. We prove by induction on $t:=|L|+|R|$ that $S$ contains at least one of $L$, $R$, or a max-induced monotonic sequence of length at least $n^\eps/2$, where $\eps = 4^{-t}$. If $t\leq 5$, then one of $L$ or $R$ has at most two elements, so the statement is clear. We therefore assume $t> 5$, in which case we may assume also that $n$ is large enough for the inequalities below.

We first describe an algorithm that shows either $S$ contains a max-induced monotonic sequence of length at least $n^{\eps}/2$ 
or there is an index $i$ such that $a_i$ is the maximum in a substantial interval, i.e.\ $a_j \leq a_i$ for $i-n^{1-\eps} \leq j \leq i+n^{1-\eps}$. 

We will attempt to build an increasing sequence from the leftmost element and a decreasing sequence starting from the rightmost element.
Initialize index sets $I = [n]$ and $F_{\lef} = F_{\ri} =\emptyset$. In the $k^{\text{th}}$ iteration of the algorithm, suppose the index $j_k \in I$ is such that $a_{j_k} = \max_{i\in I} a_i$. Let $\ell = \max F_\lef$ and $r = \min F_\ri$. For $k = 1$, we set $\ell = 0$ and $r = n+1$. Then
\begin{itemize}
    \item if $j_k - \ell< n^{1-\eps}$, add $j_k$ to $F_\lef$, delete all $i\in I$ such that $i \leq j_k$, and proceed to the next iteration;
    \item if $r - j_k < n^{1-\eps}$, add $j_k$ to $F_\ri$, delete all $i\in I$ such that $i\geq j_k$, and proceed to the next iteration;
    \item otherwise, the algorithm terminates.
\end{itemize} 
Observe that $(a_i)_{i\in F_\lef}$ and $(a_i)_{i\in F_\ri}$ are max-induced monotonic sequences in $S$,
so if either $|F_\lef| \geq n^{\eps}/2$ or $|F_\ri| \geq n^{\eps}/2$, then we have constructed a max-induced monotonic subsequence of length $n^{\eps}/2$. This must be the case if the algorithm terminates due to deleting all of $I$. 

Otherwise, suppose the algorithm terminates on the $k^{\text{th}}$ iteration for some fixed $k$. Then $j_k$ satisfies $a_i \leq a_{j_k}$ for $j_k-n^{1-\eps} \leq i \leq j_k+n^{1-\eps}$. Furthermore, for $F_{\max} := \{i \in I: a_i = a_{j_k}\}$, the constant sequence $(a_i)_{i\in F_{\max}}$ is a max-induced monotonic sequence in $S$, so we may assume $|F_{\max}| < n^\eps/2$.

Our goal now is to show that $S$ must contain a max-induced copy of either $L$, $R$, or a long monotonic sequence. Very roughly, we seek a large ``blow-up'' of either 132 or 231, where a blow-up of 231 is a subsequence on $I\cup J \cup K$ such that $(a_i, a_j, a_k)$ is a max-induced 231 pattern for any $i\in I, j\in J, k\in K$. Supposing we have such a blow-up of 231 with $I$ and $K$ large, we may apply the inductive hypothesis in $I$ with the permutation patterns $L_\lef$ and $R$ and on $K$ with the permutation patterns $L_\ri$ and $R$, where $L_\lef$ and $L_\ri$ are the permutation patterns to the left and right of the maximum element in any representative of the pattern $L$. If we find a max-induced $R$ or long monotonic sequence, we are done, so we may assume $I$ contains the pattern $L_\lef$ and $K$ the pattern $L_\ri$, which can be joined together using the 231 blow-up structure to find the pattern $L$ in $I \cup J \cup K$. We are not able to find a true blow-up of either 132 or 231, so first some care is needed to find a suitable replacement.

Let $I_\lef = \{i\in I: i< j_k\}$ and $I_\ri = \{i\in I: i> j_k\}$. Let $M\,( \supset F_{\max})$ be the indices of the largest $n^{(1-\eps)/2}$ elements in $(a_i)_{i\in I}$, breaking ties arbitrarily. Either $M_\lef := M \cap I_\lef$ or $M_\ri := M \setminus M_\lef$ has at least $\frac12 n^{(1-\eps)/2}$ elements. Suppose the former (the argument for the latter is similar). 

\begin{itemize}
\item Since $|I_\ri| \geq n^{1-\eps}$ there must be an interval $I' \subset I_\ri$ of length at least $\frac12 n^{(1-\eps)/2}$ such that $I' \cap M_\ri = \emptyset$. 
\item Since $|M_\lef| \geq \frac12n^{(1-\eps)/2}$ and $|F_{\max}| < n^\eps /2$, there is an interval $J \subset I_\lef$ such that $|J \cap M_\lef| \geq n^{(1-3\eps)/2}$ and $J\cap F_{\max} = \emptyset$. 
\end{itemize}

We are not quite done: observe that there may be an element indexed by $I'$ equal to an element indexed by $J$. To break ties, we restrict $I'$ further. 
As before, we define $I_{\max}' := \{i \in I': a_i = \max_{j\in I'} a_{j}\}$. The constant sequence $(a_i)_{i\in I_{\max}'}$ is a max-induced monotonic sequence in $S$, so we may assume $|I_{\max}'| < n^\eps/2$. Then,
\begin{itemize}
\item There is an interval $I'' \subseteq I'$ of length at least $\frac12 n^{(1-3\eps)/2}$ such that $|I'' \cap I_{\max}'| = \emptyset$.
\end{itemize}
Now let $A = J \cap M_\lef$ and $B = I''$. By construction, any max-induced subsequence in $A$ or $B$ is a max-induced subsequence of $S$. Also if $i\in A, \, j\in B$ then $a_j < a_i < a_{j_k}$. Recall that the permutation $L$ is $132$-avoiding. 
Define $L_\lef$ to be the subpermutation preceding $\max L$ and $L_\ri$ to be the subpermutation following $\max L$. Observe that $L_\lef$ and $L_\ri$ are both 132-avoiding.

We finally apply the inductive hypothesis to $S_A:= (a_i)_{i \in A}$ with the permutations $L_\lef$ and $R$, and to $S_B := (a_i)_{i\in B}$ with the permutations $L_\ri$ and $R$. If we find a max-induced copy of $R$ in $S_A$ or $S_B$ we are done. If $S_A$ contains a max-induced monotonic subsequence of length at least 
$$\frac12|A|^{4^{-(|L_\lef| + |R|)}} \geq \frac12(n^{(1-3\eps)/2})^{4^{-(t-1)}} = \frac12(n^{(1-3\eps)/2})^{4\eps} \geq \frac{n^{\eps}}{2}$$ 
or $S_B$ contains a max-induced monotonic subsequence of length at least 
$$\frac12|B|^{4^{-(|L_r| + |R|)}} \geq \frac12(n^{(1-3\eps)/2})^{4^{-(t-1)}} = \frac12(n^{(1-3\eps)/2})^{4\eps} \geq \frac{n^{\eps}}{2}$$ 
we are finished. Otherwise, there is a max-induced copy of $L_\lef$ in $S_A$ and $L_r$ in $S_B$, which along with $a_{j_k}$ forms a max-induced copy of $L$ in $S$.
\end{proof}

Given $n, k \in \mathbb N$, recall $f(k)$ is the largest integer for which there exists some $\varepsilon > 0$ such that every sequence of length $n$ contains either at least $f(k)$ distinct max-induced patterns on $k$ elements or a monotonic subsequence of length $n^\varepsilon$.
 
\subsection{Separated Subsequences}
Up to this point, our sequence results do not use a useful fact about the $\gd$-sequences which features often in stepping-up. A sequence $(a_1,a_2,\ldots,a_{k})$ has the \emph{unique maximum property} if for any interval $I$, $|\{i\in I: a_i = \max(a_i)_{i\in I}\}| = 1$. We capitalize on this useful property in Lemma~\ref{lem:spread_or_all}, which studies a second notion of subsequence that is used in our proof of Theorem~\ref{thm:step-up}\ref{thm:step-up3}. 
Call $(a_{i_1},a_{i_2},\ldots,a_{i_t})$ with $0<i_1<\dots <i_t\leq m$ a \emph{separated subsequence} of $(a_1,a_2,\ldots, a_m)$  if $i_{j+1}>i_{j}+1$  for all $j \in [t]$. 
Similar to max-inducedness, the following simple fact captures the usefulness of this definition for stepping up: suppose $v_1 < v_2 < \dots < v_{\ell+1}$ are vectors in $\{0,1\}^m$ and let $\delta_i = \delta(v_i, v_{i+1})$. If $(\gd_{i_1},\ldots, \gd_{i_k})$ is a separated subsequence of $(\gd_1,\ldots,\gd_\ell)$, then there are $v_{j_1},\ldots, v_{j_{2k}}$ such that $\gd(v_{j_{2s-1}}, v_{j_{2s}}) =  \gd_{i_s}$ for each $s \in [k]$.
Given a sequence $S= (s_1,\ldots, s_n)$, define $\|S\| := |\{s_1,\ldots, s_n\}|$, that is, the number of distinct values in the sequence. 
We can now state our key result on separated subsequences.
\begin{lemma}\label{lem:spread_or_all}
Let $A = (a_i)_{i=1}^n$ be a sequence with the unique maximum property and $k\in \N$. If $\|A\| < n^{1/(k+1)}$ and $n$ is large enough, $A$ contains every permutation pattern on $[k]$ as a separated subsequence.
\end{lemma}
The proof relies on a simple density argument to find a very rich substructure 
that contains within it all possible small structures.

\begin{proof}[Proof of Lemma~\ref{lem:spread_or_all}]
We first find constant subsequences $A_1,\ldots, A_{k}$ such that 
\begin{enumerate}[label=(\roman*)]
\item \label{size} $|A_i| \geq n^{1-i/(k+1)}$ for each $i \in [k]$,
\item \label{strat} if $a \in A_i$ (meaning $a$ is equal to some element of $A_i$) and $a' \in A_j$ with $i < j$, then $a < a'$, and 
\item \label{inter} (interlacing property) if $a_j, a_\ell \in A_i$ such that $j < \ell$, then there is some $a_m \in A_{i-1}$ such that $j < m < \ell$
\end{enumerate}

Given such subsequences, embedding an arbitrary permutation $(\sigma(1),\ldots, \sigma(k))$ is straightforward: index the elements of $A_k$ by $(a_{\ell_1},\ldots, a_{\ell_s})$. For each $j \in [k]$, use the interlacing property (and the fact $s \geq n^{1/(k+1)}$) to find a subsequence $(a_{j_1}, \ldots, a_{j_k})$ of $A_j$ with $\ell_{m2^{k+1}} < j_m < \ell_{(m+1)2^{k+1}}$. The (separated) subsequence $(a_{\sigma(1)_1},\ldots, a_{\sigma(k)_k})$ has the same pattern as $(\sigma(1),\ldots, \sigma(k))$ by property \ref{strat}.

We construct the desired subsequences by induction. Since $\|A\| < n^\frac{1}{k+1}$, there is a constant subsequence $A_1$ such that $|A_1| \geq n^{1-1/(k+1)}$. Note that by the unique maximum property, $A_1$ must be a separated subsequence. For $i \geq 1$, given a constant separated subsequence $A_i = (a_{i_j})_j$ satisfying the above properties, we construct $A_{i+1}$ as follows: 

Let $A_i' = (a_{i_j}')_j$ be the subsequence of $A$ consisting of the maximum values between the elements of $A_i$, meaning if $a_{i_j}, a_{i_{j+1}} \in A_i$, then $a_{i_j}' = \max \{a_\ell : i_j \leq \ell \leq i_{j+1}\}$. By the unique maximum property of $A$, $a_{i_j}'$ is well-defined and is strictly greater than $a_{i_j}$ and $a_{i_{j+1}}$. Then $|A_i'| = |A_i| - 1 \geq n^{1-i/(k+1)}-1$, and $\|A_i'\| \leq n^{1/(k+1)} - 1$ from the assumption that $\|A\| \leq n^{1/(k+1)}$. Thus there is a constant subsequence of $A_i'$ with length at least 
\[ \frac{|A_i'|}{\|A_i'\|} \geq n^{1-(i+1)/(k+1)},\]
and we define $A_{i+1}$ as this subsequence. It is clear from the construction that properties \ref{size}, \ref{strat}, and \ref{inter} are satisfied. 
\end{proof}

We are now ready to employ these results in our stepping-up constructions.

\section{Stepping Up with Many Colours}\label{sec:app}
In this section, we explain how the results on max-induced and separated patterns in sequences allow us to step up colourings in which all large sets of vertices span many colours.
We prove Theorem~\ref{thm:step-up} as well as Corollary~\ref{cor:cliquescolours}.
Throughout, all colourings discussed will be (hyper)edge-colourings. 
We say a colouring of a complete hypergraph with $q$ colours is \emph{$(t;q,p)$-rainbow} if every set of $t$ vertices spans at least $p$ colours. 
In this language $r_k(t;q,p)-1$ is the largest integer $n$ for which there exists a $(t;q,p)$-rainbow colouring of $\K{n}{k}$. 
Let $V(\K{n}{r+1})=\{0,1\}^n$ and given $v,w \in V$, recall that 
\[\gd(v,w) := \max \{i: v_i \neq w_i\},\]
and that $v<w$ if and only if $v_\gd < w_\gd$.
To each set of vertices $v_1< v_2 < \dots < v_{k+1}$ in $K_n^{(r+1)}$ corresponds a sequence $(\gd_1,\gd_2,\ldots,\gd_{k})$ given by $\gd_i = \gd(v_i,v_{i+1})$, which we refer to as the \textit{corresponding $\gd$-sequence}.
The crucial properties for our application of the results of Section \ref{sec:indEHfam} to hypergraph Ramsey theory are
\begin{equation}\label{unique}
\text{For an interval $I$, let } \gd_I = \max(\gd_i)_{i\in I}. \text{ Then } |\{i\in I: \gd_i = \gd_I\}| = 1.
\end{equation}
\begin{equation}\label{max}
\text{For any } v_1< v_2<\dots < v_{k+1},\, \gd(v_1,v_{k+1}) = \max\{\gd_1,\gd_2,\ldots,\gd_{k}\}.
\end{equation}
See e.g.\ \citep{grs} for proofs. 
A useful perspective is gained by viewing $[2^n]$ as the leaves of a binary tree of depth $n$, ordered from left to right as you would draw them. 
Given two leaves $v,w\in [2^n]$, the number $2\delta(v,w)$ is simply the length of the unique path between them.

 Recall that Corollary \ref{cat} shows that a sequence $(\gd_1,\gd_2,\ldots,\gd_{\ell})$ either contains many distinct max-induced patterns or a long monotonic max-induced subsequence. 
 In order to make use of this fact, we recall Claim~\ref{CorrespEdge}, which simply says the following: suppose $v_1 < v_2 < \dots < v_{\ell+1}$ are vectors in $\{0,1\}^m$ and let $\delta_i = \delta(v_i, v_{i+1})$. If $(\gd_{i_1},\ldots, \gd_{i_k})$ is a max-induced subsequence of $(\gd_1,\ldots,\gd_\ell)$, then there are $v_{j_1},\ldots, v_{j_{k+1}}$ such that $\gd(v_{j_s}, v_{j_{s+1}}) =  \gd_{i_s}$.

\begin{proof}[Proof of Claim~\ref{CorrespEdge}]
Let $v_{j_1} = v_{i_1}$ and $v_{j_{k+1}} = v_{i_k + 1}$. Then for $s>1$, define $v_{j_s} = v_{i_{s-1} + 1}$ if $\gd_{i_{s-1}} < \gd_{i_s}$ and $v_{j_s} = v_{i_s}$ otherwise. It is straightforward to check, using property (\ref{max}), that $\gd(v_{j_s}, v_{j_{s+1}}) =  \gd_{i_s}$ for each $s$.
\end{proof}
We are ready to use Theorem~\ref{thm:indEHfam} to prove the first two parts of Theorem~\ref{thm:step-up}.
\begin{proof}[Proof of Theorem~\ref{thm:step-up}\ref{thm:step-up1} and~\ref{thm:step-up2}]
Let $k,p,q \geq 3$ be such that $p \leq C_k$, the $k^{\text{th}}$ Catalan number, and let $c:=16^k+1$. 
We first prove Part~\ref{thm:step-up2}---the inequality $r_{k+1}(t^c;2q + p,p) > 2^{r_k(t;q,p) - 1}$---
by showing such that we can construct a $(t^c;2q + p,p)$-rainbow colouring $\chi'$ of $K_{2^n}^{(k+1)}$ from any $(t;q,p)$-rainbow colouring $\chi$ of $K_n^{(k)}$. Let $\chi$ be such a rainbow colouring.

We identify $V(\K{2^n}{k+1})$ and $V(\K{n}{k})$ with $\{0,1\}^n$ and $[n]$, respectively.
Partition the set of patterns on $[k]$, excluding the two monotonic permutations, into $p-2$ classes $\cC_1,\ldots, \cC_{p-2}$ such that each class contains at least one 132-avoiding permutation pattern and one 231-avoiding permutation pattern. 
Let $\cC_{p-1}$ be the family of patterns containing only the strictly increasing permutation and $\cC_{p}$ the family containing the strictly decreasing permutation.
Observe that such a partition is guaranteed to exist by Corollary \ref{cat} and the assumption that $C_{k}\geq p$ (note that a permutation may be both 132-avoiding and 231-avoiding). 
For a sequence $S$, we say $S\in \cC_i$ if the pattern of $S$ is in $\cC_i$.
Assign to each class $\cC_i$ a distinct colour $c_i$. 
Then, for an edge $e=\{v_1,\ldots, v_{k+1}\}$ of $K_{2^n}^{(k+1)}$, with corresponding $\gd$-sequence $(\gd_1,\ldots, \gd_k)$, let
\[\chi'(e) =
\begin{cases}
\chi(\{\gd_1,\ldots, \gd_{k}\})\times 1 &\text{if } \gd_1<\gd_2<\dots < \gd_{k};\\
\chi(\{\gd_1,\ldots, \gd_{k}\})\times 2 &\text{if } \gd_1>\gd_2>\dots > \gd_{k};\\
c_i &\text{if } (\gd_1,\ldots, \gd_{k})\in \cC_i,\ i \leq p-2.
\end{cases}
\]
We claim $\chi'$ is the desired colouring. 
Consider a set of $t^c$ vertices $\{v_1,\ldots,v_{t^c}\}\subset V(\K{2^n}{k+1})$ and its corresponding $\gd$-sequence $(\gd_1,\ldots, \gd_{t^c-1})$.
Since each class $\mathcal{C}_i$ contains a 132-avoiding permutation pattern and a 231-avoiding permutation pattern, Theorem \ref{thm:indEHfam}  thus tells us that $(\gd_1,\ldots, \gd_{t^c-1})$ either contains a max-induced pattern from each of $\cC_1,\dots,\cC_{p}$ or a max-induced monotonic sequence of length $t$ (the value of $c$ was chosen for this purpose).
Claim~\ref{CorrespEdge} shows that for every max-induced sequence $S\subset(\gd_1,\ldots, \gd_{t^c-1})$ of length $k$ we can find $\{v_{i_1},\dots,v_{i_{k+1}} \}\subset \{v_1,\ldots,v_{t^c}\}$ whose corresponding $\delta$-sequence is $S$. 
In other words, we have that $\chi'(\{v_{i_1},\dots,v_{i_{k+1}}\}) = c_i$ if the pattern of $S$ is in $\cC_i$. 
Therefore, $(\gd_1,\ldots, \gd_{t^c-1})$ containing a max-induced pattern from each of $\cC_1,\dots,\cC_{p}$ implies that $\{v_1,\ldots,v_{t^c}\}$ spans at least $p$ colours.

If instead $(\gd_1,\ldots, \gd_{t^c-1})$ contains a max-induced monotonic subsequence, say $(\gd_1,\ldots, \gd_t)$, then by Claim~\ref{CorrespEdge} we can find $\{v_{i_1},\dots,v_{i_{t+1}}\}\subset \{v_1,\dots,v_{t^c}\}$ whose corresponding pattern is $(\gd_1,\dots,\gd_t)$.  
We may assume without loss of generality that $(\gd_1,\ldots,\gd_t)$ is increasing, and by property (\ref{unique}), $\gd_1<\dots<\gd_t$.
Every subsequence of a strictly increasing max-induced sequence is a strictly increasing max-induced subsequence. 
Thus, again  by Claim~\ref{CorrespEdge}, for every subsequence $S\subset (\gd_1,\ldots, \gd_t)$ of length $k$ we can find a subsequence $S'\subset (v_{i_1},\dots,v_{i_{t+1}})$ of length $k+1$ whose corresponding $\gd$-sequence is $S$. 
Observe that $\chi'(S') = (\chi(S),1)$, so the set of $(k+1)$-edges on $\{v_{i_1},\dots,v_{i_{t+1}}\}$ spans at least as many colours as the set of $k$-edges on $\{\gd_1,\ldots,\gd_t\}$.
By assumption, this is at least $p$. 

All that remains is to count the number of colours used by $\chi'$ . 
This is at most $2$ colours for each of the $q$ colours used by $\chi$ plus an extra $p-2$ for the $c_i$'s.
Combined with the above arguments, this shows that $\chi'$ is a $(t^c;2q + p-2,p)$-rainbow colouring and in particular a $(t^c;2q + p,p)$-rainbow colouring. 

We now describe the construction of the $(t^c;q,p-2)$-rainbow colouring, say $\chi''$, which proves Part~\ref{thm:step-up1} i.e.\ that $r_{k+1}(t^c;q,p-2) > 2^{r_k(t;q,p-2) - 1}$. 
Let a $(t;q,p-2)$-rainbow colouring $\chi$ of $K_n^{(k)}$ be given.
Partition the set of patterns on $[k]$ as before, and assign to each class $\cC_i$ a distinct colour $c_i$ for $i\in [p-2]$, only this time identify these colours with the first $[p-2]$ colours used by $\chi$ (trivially we have $q\geq p-2$).  
For an edge $e=\{v_1,\ldots, v_{k+1}\}$ with corresponding $\gd$-sequence $(\gd_1,\ldots, \gd_{k})$, let
\[\chi''(e) =
\begin{cases}
\chi(\{\gd_1,\ldots, \gd_{k}\}) &\text{if } \gd_1<\gd_2<\dots < \gd_{k}\text{ or } \gd_1>\gd_2>\dots > \gd_{k} ;\\
c_i &\text{if } (\gd_1,\ldots, \gd_{k})\in \cC_i.
\end{cases}
\]
We omit the proof that $\chi''$ is in fact the desired colouring, as it is almost identical to the proof for  $\chi'$.
We simply note that now we can only force $p-2$ colours in any set of $t^c$ vertices, as an edge with a monotonic pattern may receive any one of the colours $c_i$ for $i\in [p-2]$.
\end{proof}

In applying our stepping-up results, we will need the following result on rainbow colorings which is proven by a standard use of the first moment method (see e.g.\ \cite{alonspencer}).
\begin{proposition}\label{fmm}
For every $k$ and $q$, there exists $\eps > 0$ and $t_0$ so that for $t > t_0$ there is a $(t;q,q)$-rainbow colouring of $K_n^{(k)}$ with $n = 2^{\eps t^{k-1}}$.
\end{proposition}
\begin{proof}
Set $\eps = 1/(q k!)$ and $t_0 = eq$. Consider a uniformly random $q$-colouring of $K_n^{(k)}$ with $n = 2^{\eps t^{k-1}}$. The probability a given $t$-set contains fewer than $q$ colours is at most $q(1-1/q)^{t \choose k}$. Therefore the expected number of $t$ sets with fewer than $q$ colours is at most
\[{n \choose t}q(1-1/q)^{t \choose k} < \left(\frac{en}{t}\right)^t q e^{-{t\choose k}/q} < \left(\frac{q^{1/t}en}{te^{-{t\choose k}/(tq)}}\right)^t < 1.\]
As the expectation is strictly less than one, there must exist some $(t;q,q)$-rainbow colouring.
\end{proof}

\begin{proof}[Proof of Corollary~\ref{cor:cliquescolours}]
Proposition~\ref{fmm} shows that there is an $\eps>0$ such that $r_3(t;q,q)> 2^{\eps t}$.
Since the Catalan number $C_i\geq 5$ for $i\geq 3$, we may apply the inequality $r_{i+1}(t^{c};2q + 5,5) > 2^{r_i(t;q,5)-1}$ from Theorem~\ref{thm:step-up},  once for each $i \in \{3,\dots,k-1\}$ and with a different $c\geq 1$ each round. 
This ultimately yields a lower bound of $r_k(t^{c'};q',5) \geq T_{k-1}(t)$ where $q' = 3^k q$ and $c'\geq 1$.
\end{proof}

 \noindent The proof of Corollary~\ref{cor:(3,3)} follows the exact same procedure as the proof of Corollary~\ref{cor:cliquescolours}, simply using the bound $r_{k+1}(t^c;3,3) > 2^{r_k(t;3,3)-1}$ in place of the bound $r_{k+1}(t^c;2q+5,3)> 2^{r_k(t;q,5) - 1}$. 
 We thus omit it. 
We now prove the final part of Theorem~\ref{thm:step-up}, which combines Lemma~\ref{lem:spread_or_all} with a slightly different stepping-up technique.

\begin{proof}[Proof of Theorem~\ref{thm:step-up}\ref{thm:step-up3}]
Let $k,p,q \geq 3$ be such that $p \leq k!$ and let $c:=k+2$. 
We prove the inequality $r_{2k}(t^c;pq,p) > 2^{r_k(t;q,p) - 1}$
by showing such that we can construct a $(t^c;pq,p)$-rainbow colouring $\chi'$ of $K_{2^n}^{(2k)}$ from any $(t;q,p)$-rainbow colouring $\chi$ of $K_n^{(k)}$. Let such a $\chi$ be given. 

We identify $V(\K{2^n}{2k})$ and $V(\K{n}{k})$ with $\{0,1\}^n$ and $[n]$ respectively.
Order the $k!$ permutations of $[k]$ and let $P_i$ be the set of sequences whose pattern is the $i^{\text{th}}$ permutation for $i \in [p]$. 
Then, for an edge $e=\{v_1,\ldots, v_{2k}\}$ with corresponding $\gd$-sequence $(\gd_1,\ldots, \gd_{2k-1})$ we  set
\[\chi'(e) =  \chi(\{\gd_1,\gd_3,\ldots, \gd_{2k-1}\}) \times i 
\]
if  $(\gd_1, \gd_3, \ldots, \gd_{2k-1})\in P_i$ for some $i\in [p]$ and we let $\chi'(e)$ be arbitrary otherwise. 

We now check that this is the desired colouring. 
It is clear that it uses at most $pq$ colours.
Suppose $v_1<\dots< v_{t^c}$ are vertices of $\K{2^n}{2k}$ and let $(\gd_1,\dots,\gd_{t^c-1})$ be the corresponding $\gd$-sequence. 
If $\|(\delta_i)_{i=1}^{t^c-1}\| \geq 2t$ then we can find an index set $I \subset [t^c - 1]$ such that $(\delta_i)_{i\in I}$ is a separated subsequence with $\|(\delta_i)_{i\in I}\| \geq t$. 
Let $f = \{\delta_{i_1},\dots,\delta_{i_k}\}\subset (\gd_i)_{i \in I}$ where $i_1<i_2<\dots <i_k$. 
Then the edge $e = \{v_{i_1},v_{i_1+1},v_{i_2},v_{i_2+1},\dots, v_{i_k},v_{i_k+1}\}$ has colour $\chi'(e)= (\chi(f),\cdot)$.
As $f$ was arbitrary we have that $\{v_1, \dots, v_{t^c}\}$ spans at least as many colours as $\{\gd_i\}_{i \in I}$.
But $\|(\delta_i)_{i\in I}\| \geq t$ and so $(\gd_i)_{i \in I}$ spans at least $p$ colours by assumption.

Now suppose that $\|(\delta_i)_{i=1}^{t^c-1}\|< 2t$. 
By \eqref{unique}, $(\delta_i)_{i=1}^{t^c-1}$ satisfies the unique maximum property. 
Therefore, by our choice of $c$ and assumption that $t$ is large, we can apply Lemma~\ref{lem:spread_or_all} with $A=(\delta_i)_{i=1}^{t^c-1}$ to conclude that $A$ contains every permutation of $[k]$ as a separated subsequence. 
Suppose we have a separated subsequence $S \subset A$ whose pattern is the $i^{\text{th}}$ permutation for some $i \in [p]$. 
As before we can find a $(2k)$-edge $e\subseteq  \{v_1,\dots,v_{t^c}\}$ for whom the corresponding $\gd$-pattern is $S$.
Thus $\chi'(e) = (\cdot,i)$.
Repeating this for the first $p$ permutations implies that $\{v_1,\dots,v_{t^c}\}$ spans at least $p$ colours.
\end{proof}

Both of our stepping-up constructions (like the original) rely on breaking the stepping-up into two parts: a \textit{lifted} colouring from a lower uniformity graph and a \textit{new} colouring. 
The constructions have the property that every large set of vertices satisfies one of the following: either it spans a colouring that is \textit{lifted} from a large clique on the lower graph or it sees many distinct \textit{new} colours.
The number of new colours one can guarantee using this approach depends strongly on the uniformity of the hypergraphs in question. We ask a question in the concluding remarks concerning this.

\section{Hedgehogs}\label{sec:hedgehogs}

In this section, we relate our results on many-coloured Ramsey numbers of complete hypergraphs to a Ramsey problem on a certain class of hypergraphs which we call \textit{generalized hedgehogs}. We note our idea is very much inspired by the Conlon-Fox-R\"{o}dl construction described in the introduction. 
The main result is Theorem~\ref{thm:hedgehogs_upper}\ref{itm:hedgehog_upper} and we also prove Lemma~\ref{lem:hedgehogs_construction} for a more general family of hedgehogs.

We define the \emph{generalized hedgehog} 
$\He{t}{k}{s}$ for 
$s,k,t\in \N$ with $k > s$ and $t\geq s$ to be the following $k$-uniform hypergraph: fix a set of $t$ vertices called the \emph{body}. The edge set of $\He{t}{k}{s}$ consists of the following edges: for every $s$-subset $S$ of the body, we add a $k$-edge $e$ containing $S$ along with $k-s$ additional vertices that are contained in no other edge of $\He{t}{k}{s}$.

The $k$-uniform balanced hedgehog $\hat H_{t}^{(k)}$ is $\He{t}{k}{\lceil \frac k2 \rceil}$ and the hedgehog introduced by Conlon, Fox, and R\"odl is $\He{t}{k}{k-1}$.
We note that $\He{t}{k}{s}$ has ${t\choose s}$ edges and $t+ (k-s){t\choose s}$ vertices.

Lemma~\ref{lem:hedgehogs_construction} is a corollary of the following result, which we state in its full generality. 
\begin{lemma}\label{lem:gen_hedgehogs}
Given $k,q,p',t\in\N$, let $p={k \choose s}$ and $q' = {q\choose p}$. Then 
\[
 r_{k}(\He{t}{k}{s}; q', p'+1) > r_{s}(t;q,p'p+1) - 1.
\]
\end{lemma}
\begin{proof}[Proof of Lemma~\ref{lem:hedgehogs_construction}]
Apply Lemma~\ref{lem:gen_hedgehogs} with $2k+1, k+1$ and $1$ playing the roles of $k,s$ and $p'$ respectively. 
\end{proof}

\begin{proof}[Proof of Lemma~\ref{lem:gen_hedgehogs}]
We prove $ r_{k}(\He{t}{k}{s};q',p'+1) > r_{s}(t;q,p'p+1) -1$ by showing that given a $(t;q,p'p+1)$-rainbow colouring of $\K{n}{s}$, we can construct a $q'$-colouring $\chi'$ of $\K{n}{k}$ in which every copy of $\He{t}{k}{s}$ spans at least $p'+1$ colours. 
Let such a $\chi$ be given. 
Identify the vertex sets of $\K{n}{k}$ and $\K{n}{s}$. 
We colour $e'\in E(\K{n}{k})$ by $\chi'(e') = \{ \chi(e): e\in E(K_n^{(s)}[e'])\}$, that is, the set of all colours of the $s$-edges contained in $e'$. 
The number of $s$-edges contained in $e'$ is exactly $p={k\choose s}$, so $\chi'(e')$ will be a set of at most that many colours. 
If $|\chi'(e')|< p$ then we add arbitrary colours to the set until $|\chi'(e')|=p$.
Thus the new colouring uses at most $q'={q\choose p}$ colours.

Now consider a copy $H\subset \K{n}{k}$ of $\He{t}{k}{s}$.
As the body contains $t$ vertices, by assumption it spans at least $p'p+1$ colours under $\chi$. 
Each of these colours appears as an element of $\chi'(e')$ for some $e' \in E(H)$, i.e. $|\cup_{e'\in E(H)}\chi'(e')|\geq p'p+1$. 
However, $p$ colours appear in $\chi'(e')$ for a given $e'$, so by pigeonhole there must be more than $p'$ edges $e'\in E(H)$ with distinct colours in $\chi'(H)$. 
\end{proof}

We now derive Theorem~\ref{thm:hedgehogs_upper}\ref{itm:hedgehog_lower}.

\begin{proof}[Proof of Theorem~\ref{thm:hedgehogs_upper}\ref{itm:hedgehog_lower}]
Throughout the proof we assume $t\in \N$ is sufficiently large. 
Let $m = \lceil\frac{e(k+1)}{\log_2 (k+1)}\rceil$, which is chosen so that $m! > \left(\frac{m}{\mathrm{e}}\right)^m > 2^{2k} > {2k+1\choose k+1}$ for $k$ sufficiently large.
Recall that by Proposition~\ref{fmm}, $r_{m}(t;m!,m!)\geq T_2(\eps t)$ for some $\eps>0$.
We now use the inequality $r_{2k}(t^c;pq,p) \geq 2^{r_k(t;q,p)}$ from Theorem~\ref{thm:step-up} exactly $\lfloor \log_2 \frac{k+1}{m} \rfloor > \log_2\log_2 k - 3$ times, to obtain that there exists $q'$ and $c' \geq 1$ such that $r_{k+1}(t^{c'};q',m!) \geq T_{\log_2\log_2 k -1}(t)$ (if $k+1 > 2^j m$ with $j = \lfloor \log_2 \frac{k+1}{m} \rfloor$ we use the simple observation $r_{i+1}(t; q, p) > r_{i}(t;q',p)$ for some $q' > q$).
This is possible because the uniformity increases at each step while the number of colours imposed remains at $m!$.   
Since $m! > {2k+1\choose k+1}$, we can apply Lemma~\ref{lem:hedgehogs_construction} to get that for some $q''$, $r_{2k+1}(\hat H_{t};q'',2)\geq T_{\log_2\log_2 k - 2}(t)$, where we removed the exponent $c'$ at the cost of a tower height. 
We choose $c>0$ such that if $m! \leq {2k+1\choose k+1}$ then $c\log_2\log_2k \leq 1$.

Thus, for all $k$, $r_{2k+1}(\hat H_{t};q'',2)\geq T_{c\log_2\log_2 k }(t)$. 
Furthermore, it is clear that $q''$ is only a function of $k$ and of $m$, which is itself a function of $k$, as required.
\end{proof}

Conlon, Fox, and R\"odl showed the following using a similar argument to Theorem~\ref{thm:hedgehogs_upper}\ref{itm:hedgehog_upper}: 
\begin{theorem}[\cite{CFR}]
For all $k \geq 4$ there exists $c>0$ such that 
\[
r_k(H_t^{(k)};2,2) \leq T_{k-2}(ct).
\]
\end{theorem}
For the case $k=4$, a construction of Kostochka and R\"odl \cite{KR} shows this is approximately sharp, i.e. that $r_4(\He{t}{4}{3};2,2) = T_2(\Omega(t))$.
We cannot prove a matching lower bound for $k=5$ as when we attempt to apply  Lemma~\ref{lem:hedgehogs_construction} we need to impose $6$ colours on $4$-uniform graphs. 
When stepping up to uniformity $4$, Theorem~\ref{thm:step-up} allows us to impose at most $C_3=5$ colours, so we cannot beat the random argument here.
We do, however, obtain the following:
\begin{lemma}
For $ 5\leq  k \leq 13 $, there exist  $c>0$ and $q\in \N$  such that for all $t$ 
\[
r_k(H_t^{(k)};q,2) \geq T_{k-3}(t^c).
\]
\end{lemma}
\begin{proof}
The proof mimics that of Theorem~\ref{thm:hedgehogs_upper}\ref{itm:hedgehog_lower}, but now we start from the fact that $r_4(t;14,14)\geq T_2(ct)$ (using Proposition~\ref{fmm}) and leverage that $C_4=14$. 
Starting from $\ell =4$, we apply the relation $r_{\ell+1}(t^C;2q+p,p) \geq 2^{r_{\ell}(t;q,p)}$  from Theorem~\ref{thm:step-up} $k-5$ times. 
This gives us $r_{k-1}(t;q',14)\geq T_{k-3}(t^{c'})$ for some $q'\in \N$ and $c'\geq 1$. 
Applying Lemma~\ref{lem:hedgehogs_construction} then gives the result using $q''$ colours and with $c'$ as $c$. 
\end{proof}

We now prepare to prove our upper bounds on the Ramsey numbers of balanced hedgehogs, Theorem~\ref{thm:hedgehogs_upper}\ref{itm:hedgehog_upper}.
Let $H$ be an $r$-uniform hypergraph and $A\subset V(H)$ with $|A|< r$. 
The \emph{piercing number} of $A$ (denoted $\tau_H(A)$) is the size of the smallest set of vertices from $V(H)\backslash A$ that intersects every edge of $H$ containing $A$. 
Equivalently it is the minimum number of vertices that must be deleted from $V(H)\backslash A$ to delete all edges containing $A$. 

\begin{proposition}\label{fact:deg_gives_matching}
Suppose $H$ is an $r$-uniform hypergraph and $v\in V(H)$ has $\tau(v) \geq (r-1)m$.
Then there exist $m$ edges whose pairwise intersections are all precisely $v$. 
\end{proposition}
\begin{proof}[Proof of Proposition~\ref{fact:deg_gives_matching}]
Let $X\subset V(H)$ be a witness to $\tau(v)$.
It is clear the set of edges incident to $v$ has order at least $(r-1)m$. 
As each such edge $e$ contains at most $r-1$ elements of $X$, and $|X|$ is minimal with respect to intersecting all edges incident to $v$, we can greedily find the desired $m$ edges.
\end{proof}

\begin{proof}[Proof of Theorem~\ref{thm:hedgehogs_upper}\ref{itm:hedgehog_upper}]
Let $G = \K{n}{2k+1}$ where $n \geq t^{k+3}$ and let $\chi$ be an arbitrary red/blue colouring of the edges of $G$. 
We will show that there exists a monochromatic copy of $H_t^{(2k+1)}(k+1)$ in $G$, and thus that $\R{2}{2k+1}(\He{t}{2k+1}{k+1};2,2) \leq n$.
We will do so by using $\chi$ to define a partial edge $2$-colouring $\chi'$ of $G' :=\K{n}{k+1}$ and in turn using $\chi'$ to define a $2$-colouring $\chi''$ of $V(G)$.
We will then find a large monochromatic set of vertices in $\chi''$, use this to find a large (red or blue) independent set in $G'$, and finally use this to find our monochromatic $\He{t}{2k+1}{k+1}$ in $G$. 

Throughout the proof let $H_t: =\He{t}{2k+1}{k+1}$.
We say a set of vertices is in {\em red danger} if its red piercing number (its piercing number in the subgraph of red edges) is less than $t^{k+1}$ and similarly define {\em blue danger}. Then for $e \in G':= \K{n}{k+1}$, let $\chi'(e)$ be red if $e$ is in red danger and blue if $e$ is in blue danger. 
As $\tau_G(e) > 2t^{k+1}$, we have that this partial colouring $\chi'$ is well defined. 

Now say a vertex $v$ is in {\em red peril} if its red piercing number in $G'$ (under $\chi'$) is at most $2kt^{k+1}$ and similarly define {\em blue peril}. Let $\chi''(v)$ be red if $v$ is in red peril and blue if $v$ is in blue peril, with ties broken arbitrarily.
We claim that $\chi''$ assigns a colour to every vertex. 
Indeed suppose that some vertex $v$ has red piercing number and blue piercing number at least $2k t^{k+1}$ under $\chi'$.

By Proposition~\ref{fact:deg_gives_matching}, there exist $(k+1)$-edges $e_1, \dots, e_s, f_1, \dots, f_s \in G'$ where $s = 2t^{k+1}$ such that 
\begin{itemize}
    \item $e_i \cap e_j = f_i \cap f_j = \{v\}$ for all $i \neq j$,
    \item $\chi'(e_i)$ is red for all $i \in [s]$, and
    \item $\chi'(f_j)$ is blue for all $j \in [s]$
\end{itemize}
Let $A_1, A_2, \dots, A_s$ be disjoint subsets each of size $k-1$ in $V(G) \setminus \left(\bigcup_{i=1}^s (e_i \cup f_i)\right)$.
For each $i, j \in [s]$, define a $(2k+1)$-edge $g_{ij}$ to be $e_i \cup f_j$ along with an arbitrary choice of $(2k+1)-|e_i \cup f_j|$ vertices from $A_{(i+j-1)\mathrm{mod}\ s}$. Observe that $e_i \cup f_j \neq e_{i'} \cup f_{j'}$ for $(i,j) \neq (i',j')$, so in particular, these edges are all distinct.

At least half of these edges must have been, without loss of generality, red under the colouring $\chi$ of $G$. 
Therefore, there is some $i\in [s]$ such that $e_i$ is contained in at least $\frac{s^2}{2s} =  \frac{s}{2} > t^{k+1}$ red $(2k+1)$-edges under $\chi$.
But since $g_{ij} \cap g_{ij'} = e_i$ for $j \neq j'$, this contradicts the fact that $e_i$ was a red-danger edge of $G'$.
Thus, the colouring $\chi''$ of $V(G)$ colours every vertex. 

We now choose a set $X$ of red vertices (without loss of generality) which has order $\frac n2$. 
By the definition of $\chi''$, we have that the red piercing number under $\chi'$ of each $v\in X$ is at most $2kt^{k+1}$.
Therefore we can greedily find a set $Y$ of order $\frac{n}{2\cdot 2kt^{k+1}}\geq t$ which contains no red edges of $G'$ (here we use that $t$ is large relative to $k$).

All edges of $G'$ in $Y$ are not in red danger and thus have red piercing number at least $t^{k+1}$ under $\chi$.
By the definition of piercing number and as $t^{k+1} >  t +  k{t \choose k+1} = |V(\He{t}{2k+1}{k+1})|$, we can build the hedgehog greedily using $Y$ as the body.
\end{proof}

\subsection{The Burr-Erd\H{o}s Conjecture in hypergraphs}
The {\em degeneracy} of a hypergraph is the minimum $d$ such that every induced subgraph contains a vertex incident to at most $d$ edges; such a hypergraph is called {\em $d$-degenerate}.
Burr and Erd\H{o}s conjectured that for every $d$, there exists a constant $c_d \geq 1$ such that every $d$-degenerate graph $G$ on $n$ vertices satisfies $r_2(G) < c_d n$~\cite{BE}. This was finally proven by Lee in~\cite{Lee}, building on the previous work of several authors (\cite{FSburr},\cite{KS}). In the case of hypergraphs, however, the conjecture fails: Kostochka and R\"{o}dl~\cite{KR} showed $r_4(H_t^{(4)}) \geq 2^{ct}$, and were able to construct for every large enough $d$, $d$-degenerate $3$-uniform hypergraphs on $n$ vertices whose ramsey number is at least $n^{d^{1/4}}$. If one allows $3$ or more colours, Conlon, Fox, and R\"{o}dl~\cite{CFR} proved $r_3(H_t^{(3)}; 3) \geq \Omega(t^3/\log^6 t)$. Note that the degeneracy of hedgehogs is $1$.

This shows that the Burr-Erd\H{o}s Conjecture fails for $k$-uniform hypergraphs where $k\geq 4$ and for $3$-uniform hypergraphs provided the number of colours is at least $3$ or the degeneracy is large enough. For the sake of completeness, we give a simple construction of $8$-degenerate $3$-uniform  hypergraphs whose Ramsey number is not linear in the number of vertices. 

\begin{proposition}\label{nbe}
There exists a $3$-uniform hypergraph on $Cn^2$ vertices which is $8$-degenerate and for which the $2$-colour Ramsey number is at least $Cn^{3}.$ 
\end{proposition}

\begin{proof}
Let $V$ be a set of $n$ vertices, and with $m = {n \choose 2}$, let $B$ be a set of $m + 1$ vertices disjoint from $V$ with a total ordering $x_1,x_2,\ldots x_m, x_{m+1}$. Consider an ordering on the edges $E$ of the complete graph on $V$, say $e_1,e_2,...,e_m$. Let $H$ be a 3-uniform hypergraph consisting of the following: for each edge $e_i=(x,y)\in E$, add the two triples $\{x,y,x_i\}$ and $\{x,y,x_{i+1}\}$,
and for every $i \in [m-3]$, add all triples within $\{x_i,x_{i+1},x_{i+2},x_{i+3},x_{i+4}\}$. By considering the smallest element of $B$ in each subset of vertices, we see that $H$ is $8$-degenerate.

We claim $r_3(H) \geq n^3/8$. Indeed, given $K_{n^3/8}^{(3)}$, let $W_1,\ldots W_{n/4}$ be an arbitrary partition of the vertices where $|W_i|=n^2/2$ for all $i\in [n/4]$.
We colour blue all edges inside any $W_i$ and all edges with vertices in distinct sets $W_{i_1},W_{i_2},W_{i_3}$. We colour red all edges with exactly two vertices in a set $W_i$, for some $i\in [n/4]$. 

Suppose there is a red copy of $H$. In particular, the edges induced by $S =\{x_1,x_2,x_3,x_4,x_5\}$ are all red.  If $S$ intersects distinct parts $W_i,W_j$, and $W_k$ for some $i,j,k \in [n/4]$, then the vertices in the intersections form a blue edge. However, if $S \subseteq W_i \cup W_j$, then $|S \cap W_i| \geq 3$ or $|S \cap W_j| \geq 3$, which again forms a blue edge. So we have a contradiction.

Suppose instead there is a blue copy of $H$. By pigeonhole, there are two vertices $v,w \in V$ which both lie in the same $W_i$. Let $e_j = \{v,w\}$. The edges $\{v,w,x_j\}$ and $\{v,w,x_{j+1}\}$ are both blue, so $x_j,x_{j+1} \in W_i$. But $\{x_j,x_{j+1},x_{j+2}\}$ and $\{x_{j-1}, x_j,x_{j+1}\}$ are also blue, so this implies $x_{j-1}, x_{j+2} \in W_i$. Proceeding inductively, we conclude that all of $B$ must lie in $W_i$. 
 
However, $|W_i| = n^2/2$ and $|B| = {n \choose 2}$, so there are $n/2$ vertices of $V$ outside of $W_i$. Two of these vertices $z$ and $u$ must lie in the same $W_j$ with $j \neq i$. Letting $e_k = \{z,u\}$, we get that the edge $\{z,u,x_k\}$ is coloured red, a contradiction which finishes the proof.
\end{proof}

\section{Concluding remarks}\label{sec:conc}
We have proved that for every positive integer $h$, there exist $q$, $k$, and an infinite family of $k$-uniform hypergraphs whose $2$-colour Ramsey numbers differ by a tower of height $h$ from the $q$-colour Ramsey numbers.
This reinforces the fact that the number of colours plays an important role in the behaviour of Ramsey numbers of hypergraphs and casts a shadow on Erd\H{o}s's conjecture on the $2$-colour Ramsey number of a $3$-uniform clique.

Observe that both of our new stepping-up constructions rely on a dichotomy: either we can find many suitable substructures within the $\delta$-sequences (which give rise to many colours) or we must have a long monotonic sequence (which allows us to use induction). Since for every $k$-edge there are at most $k!$ distinct permutations, our methods fail to give good lower bounds for $r_k(t; q,p)$ whenever $k\ll p$. Even in the simplest case $r_3(t; q,3)$, we were not able to prove a double exponential lower bound, leaving open the following question of Conlon, Fox, and R\"odl on $r_3(t; q,3)$.
\begin{problem}\cite[Problem 1]{CFR}\label{problem:CFR}
Is there an integer $q$, a positive constant $c$, and a $q$-colouring of the $3$-uniform hypergraph on $2^{2^{ct}}$ vertices such that every subset of order $t$ receives at least $3$ colours?
\end{problem}

We propose here a much weaker problem than Problem~\ref{problem:CFR} which we were not able to resolve. We note that a negative answer would uncover a radical new phenomenon in the Ramsey numbers of hypergraphs. 
\begin{problem}
Does there exist $k\in \N$ such that the following holds? For all $p\in \N$ there exist $q\in \N$ and $c>0$ such that $r_k(t;q,p)\geq 2^{2^{t^c}}$ for all $t$ sufficiently large.
\end{problem}

\noindent A similar but much more ambitious problem was posed in \cite{setramsey}.
\begin{problem}\cite[Problem 6.3]{setramsey}\label{setramsey}
Determine the tower height of $R_k(n; r, r - 1)=r_k(n; r,r)$ for all $k \geq 3$ and $r \geq 2$.    
\end{problem}

\noindent The authors of \cite{setramsey} note the apparent difficulty of Problem~\ref{setramsey} and ask the following weaker question. Is there a fixed integer $c$ such that $R_k(n;r,r-1) \geq T_{k-c}(n)$ for every $k\geq 3$ and $r\geq 2$? We cannot answer this question, but using Theorem~\ref{thm:step-up}\ref{thm:step-up1}, we can prove $R_k(n;r, r-1)$ is at least a tower of height roughly $k - 0.5\log_2 r$. Any improvement beyond this bound would likely be very interesting.

We make the following conjecture regarding the Ramsey numbers of $k$-uniform hedgehogs. 
This would in particular demonstrate that the $2$-colour and $q$-colour Ramsey numbers of these hedgehogs, unlike those of balanced hedgehogs, do not differ by arbitrarily large tower heights.

\begin{conjecture}
There is $\ell\in \mathbb{N}$ such that for every positive integer $k$, for every sufficiently large $t$,
$$ r_k(H_t^{(k)})\geq T_{k-\ell}(t). $$
\end{conjecture}

Finally, recall Proposition~\ref{nbe} shows there is an infinite family of $3$-uniform hypergraphs which are $8$-degenerate and for which the the $2$-colour Ramsey numbers grows faster than linear in the order of the hypergraphs. It would be interesting to improve the quantitative aspects of this result.

\begin{problem}\label{deg}
Give an infinite family of $1$-degenerate $3$-uniform hypergraphs whose $2$-colour Ramsey number is not polynomial in the order of the hypergraph. 
\end{problem}
\noindent Kostochka and R\"odl \cite{KR} indicate that this may not be possible. We share this suspicion, noting that the following could be true: for any $d$, there is $f(d)$ such that the Ramsey number of any $d$-degenerate $3$-uniform hypergraph (on $n$ vertices) is at most $n^{f(d)}$.

\section*{Acknowledgments}
We thank Bhargav Narayanan for helpful conversations on this topic and the anonymous referees for their detailed and invaluable feedback, especially regarding the organization of the first two sections.
The first author was supported in part by Simons Foundation grant 332622. The first and fourth authors were supported in part by NSF Grant CCF-1814409 and NSF Grant DMS-1800521. The second author was supported by Deutsche Forschungsgemeinschaft (DFG, German Research Foundation) under Germany’s Excellence Strategy
EXC-2181/1 - 390900948 (the Heidelberg STRUCTURES Cluster of Excellence) and by EPSRC grant EP/V007327/1. The third author was supported in part by the Deutsche Forschungsgemeinschaft (DFG, German Research Foundation) -- 428212407.

\bibliographystyle{abbrv}
\bibliography{ram.bib}

\end{document}